\documentclass[11pt]{amsart}

\usepackage[T1]{fontenc}
\usepackage{lmodern}
\usepackage{microtype}
\usepackage{amsmath,amssymb,mathtools}
\usepackage{xcolor}
\definecolor{internalblue}{HTML}{1F5A94}
\definecolor{citationburgundy}{HTML}{8A2942}
\usepackage[colorlinks=true]{hyperref}
\hypersetup{
  linkcolor=internalblue,
  citecolor=citationburgundy,
  urlcolor=citationburgundy,
  pdftitle={A topological proof that compact Hausdorff spaces are not finitely co-concrete},
  pdfauthor={Marco Abbadini},
  pdfkeywords={compact Hausdorff spaces, finite concreteness, faithful functor preserving directed colimits}
}

\newcommand{\CompHaus}{\mathbf{CompHaus}}
\newcommand{\Set}{\mathbf{Set}}
\newcommand{\N}{\mathbb N}
\newcommand{\cK}{\mathcal K}
\newcommand{\op}{\mathrm{op}}
\newcommand{\im}{\operatorname{im}}
\DeclareMathOperator*{\colim}{colim}

\newtheorem{theorem}{Theorem}[section]
\newtheorem{lemma}[theorem]{Lemma}
\newtheorem{corollary}[theorem]{Corollary}
\theoremstyle{definition}
\newtheorem{definition}[theorem]{Definition}
\theoremstyle{remark}

\title[Compact Hausdorff spaces are not finitely co-concrete]
  {A topological proof that compact Hausdorff spaces are not finitely co-concrete}
\author{Marco Abbadini}
\address{Institut de Recherche en Math\'ematique et Physique,
  Universit\'e catholique de Louvain, Belgium}
\urladdr{\url{https://marcoabbadini-uni.github.io}}


\begin{document}

\begin{abstract}
Lieberman, Rosick\'y, and Vasey proved that
\(\CompHaus^{\op}\)---the opposite of the category of compact Hausdorff spaces---is not finitely concrete by a route through Hilbert and
Banach spaces, commutative unital \(C^*\)-algebras, and Gelfand duality.
We give a short topological proof.
Moreover, we strengthen the result by identifying a specific sequential colimit in \(\CompHaus^{\op}\) that no faithful set-valued functor preserves.
\end{abstract}

\maketitle

\section{Introduction}

Following Lieberman--Rosick\'y--Vasey
\cite[Definition~1]{LiebermanRosickyVasey}, we recall the relevant notion.

\begin{definition}
A category \(\cK\) is \emph{finitely concrete} if there is a faithful
functor \(U\colon\cK\to\Set\) that preserves every directed colimit that
exists in \(\cK\).
\end{definition}

For example, the category \(\mathrm{Mod}(T)\) of models and homomorphisms of a
first-order theory \(T\) is finitely concrete.

Lieberman--Rosick\'y--Vasey proved that \(\CompHaus^{\op}\) is not finitely
concrete by passing through Hilbert spaces, Banach spaces, commutative unital
\(C^*\)-algebras, and Gelfand duality
\cite[Theorem~18, Example~21(2), and the proof of
Theorem~22]{LiebermanRosickyVasey}.

The aim of this note is to give a direct topological proof that $\CompHaus^\op$ is not finitely concrete; moreover, we identify a specific sequential colimit in \(\CompHaus^{\op}\) that no faithful set-valued functor preserves.

We denote by $\N$ the set of natural numbers (including $0$), and by $\N_{>0}$ the set $\N \setminus \{0\}$.
We index $n$-tuples with $\{1, \dots, n\}$ and infinite sequences with $\N_{>0} =\{1, 2, 3, \dots\}$.

For each \(n\in\N\), define
\begin{align*}
  \begin{aligned}
    p_n\colon [0,1]^{n+1}&\longrightarrow [0,1]^n,\\
    (x_1,\dots,x_n,x_{n+1})&\longmapsto (x_1,\dots,x_n)
  \end{aligned}
  &\qquad
  \begin{aligned}
    \pi_n\colon [0,1]^{\N_{>0}}&\longrightarrow [0,1]^n,\\
    (x_1,x_2,\dots)&\longmapsto (x_1,\dots,x_n).
  \end{aligned}
\end{align*}
In \(\CompHaus^{\op}\), the sequence
\begin{equation}\label{eq:diagram}
  [0,1]^0 \xrightarrow{p_0^{\op}} [0,1]^1\xrightarrow{p_1^{\op}}[0,1]^2
   \xrightarrow{p_2^{\op}}[0,1]^3
   \xrightarrow{p_3^{\op}}\cdots
\end{equation}
has colimit \([0,1]^{\N_{>0}}\), with structure maps
\(\pi_n^{\op}\colon [0,1]^n\to [0,1]^{\N_{>0}}\).  Our main result says that, for every faithful
\(U\colon\CompHaus^{\op}\to\Set\), the comparison map for
\eqref{eq:diagram} is not surjective.

The proof rests on a finite zigzag in \([0,1]^{\N_{>0}}\).  We isolate it first.

\section{The fiber zigzag}

Define
\begin{align*}
  \delta\colon [0,1]^{\N_{>0}}& \longrightarrow [0,1]\\
  x & \longmapsto \sum_{k=1}^{\infty}2^{-k}x_k .
\end{align*}
The range of \(\delta\) is contained in \([0,1]\), since
\(0\leq \delta(x)\leq\sum_{k=1}^{\infty}2^{-k}=1\).  Its partial sums are
continuous and converge uniformly, since the tail after the \(N\)-th term is
at most \(2^{-N}\); hence \(\delta\) is continuous.

For a map \(f\) and points \(x,y\) in its domain, write
\[
  x\sim_f y \quad\Longleftrightarrow\quad f(x)=f(y).
\]
Let \(\underline 0=(0,0,\ldots)\) and
\(\underline 1=(1,1,\ldots)\) be the constant points of \([0,1]^{\N_{>0}}\).

\begin{lemma}[Fiber zigzag]\label{lem:zigzag}
For every \(n\in\N\), the points \(\underline 1\) and \(\underline 0\)
are connected by a finite zigzag in \([0,1]^{\N_{>0}}\) whose successive steps alternate
between \(\sim_{\pi_n}\) and \(\sim_\delta\).
\end{lemma}

\begin{proof}
Set
\[
  L_n=\sum_{k=1}^n2^{-k}=1-2^{-n},
  \qquad M=2^{n+1}.
\]
For \(0\leq r\leq M\), let
\[
  z_r=
  \bigl(
    \underbrace{1-r/M,\ldots,1-r/M}_{n\ \mathrm{coordinates}},
    0,0,\ldots
  \bigr),
\]
and, for \(0\leq r<M\), let
\[
  w_r=
  \bigl(
    \underbrace{1-(r+1)/M,\ldots,1-(r+1)/M}_{n\ \mathrm{coordinates}},
    L_n,0,0,\ldots
  \bigr),
\]
where \(L_n\) occupies coordinate \(n+1\).  All these points belong to
\([0,1]^{\N_{>0}}\).  Since \(2^{-(n+1)}=1/M\), for \(0\leq r<M\) we have
\[
  \begin{aligned}
  \delta(z_r)
    &=L_n\left(1-\frac rM\right),\\
  \delta(w_r)
    &=L_n\left(1-\frac{r+1}{M}\right)+2^{-(n+1)}L_n\\
    &=L_n\left(1-\frac rM\right).
  \end{aligned}
\]
Thus \(z_r\sim_\delta w_r\), while the definitions give
\(\pi_n(w_r)=\pi_n(z_{r+1})\).  Finally,
\[
  \pi_n(\underline 1)=\pi_n(z_0),
  \qquad z_M=\underline 0.
\]
Consequently,
\[
  \underline 1\sim_{\pi_n}z_0,\qquad
  z_r\sim_\delta w_r\sim_{\pi_n}z_{r+1}
  \quad(0\leq r<M),\qquad
  z_M=\underline 0.
\]
Together these relations form the required alternating zigzag.
\end{proof}

\section{The fixed sequential-colimit obstruction}

We first spell out the colimit in \eqref{eq:diagram}.  In \(\CompHaus\),
\[
  [0,1]^0\xleftarrow{p_0} [0,1]^1\xleftarrow{p_1}[0,1]^2
   \xleftarrow{p_2}[0,1]^3
   \xleftarrow{p_3}\cdots
\]
has limit \([0,1]^{\N_{>0}}\), with limit maps \(\pi_n\).  Indeed,
\([0,1]^{\N_{>0}}\) is
compact Hausdorff by Tychonoff's theorem, and a compatible family of maps
\(Y\to [0,1]^n\) determines, coordinate by coordinate, a unique map
\(Y\to [0,1]^{\N_{>0}}\),
which is continuous by the definition of the product topology.  Passing to
opposite categories yields \eqref{eq:diagram} and its stated colimit.

For any functor \(U\colon\CompHaus^{\op}\to\Set\), denote by
\begin{equation}\label{eq:comparison}
  \kappa_U\colon
  \colim_{n\in\N}U([0,1]^n)\longrightarrow U([0,1]^{\N_{>0}})
\end{equation}
the canonical comparison map induced by the functions
\(U(\pi_n^{\op})\).  Its image is
\[
  \im(\kappa_U)
  =\bigcup_{n\in\N}\im\bigl(U(\pi_n^{\op})\bigr),
\]
because every element of a colimit of sets has a representative at one of
the stages.

\begin{theorem}\label{thm:main}
If \(U\colon\CompHaus^{\op}\to\Set\) is faithful, then the comparison map
\(\kappa_U\) in \eqref{eq:comparison} is not surjective.  In fact,
\[
  \im\bigl(U(\delta^{\op})\bigr)\nsubseteq\im(\kappa_U).
\]
\end{theorem}

\begin{proof}
Let \(1\) denote the one-point space, and let
\[
  e_0,e_1\colon 1\longrightarrow [0,1]
\]
be the endpoints of \([0,1]\).  They induce distinct arrows
\[
  e_0^{\op},e_1^{\op}\colon [0,1]\longrightarrow 1
\]
in \(\CompHaus^{\op}\).  Since \(U\) is faithful, the functions
\[
  U(e_0^{\op}),U(e_1^{\op})\colon U([0,1])\longrightarrow U(1)
\]
are distinct.  Hence there is \(v\in U([0,1])\) such that
\begin{equation}\label{eq:distinguish}
  U(e_0^{\op})(v)\neq U(e_1^{\op})(v).
\end{equation}

Set
\[
  u \coloneqq U(\delta^{\op})(v)\in U([0,1]^{\N_{>0}}).
\]
Suppose, towards a contradiction, that \(u\in\im(\kappa_U)\).  Then there
are \(n\in\N\) and \(w\in U([0,1]^n)\) such that
\begin{equation}\label{eq:finite-stage}
  u=U(\pi_n^{\op})(w).
\end{equation}
For an element $x \colon 1 \to [0,1]^{\N_{>0}}$, set
\[
  a_x \coloneqq U(x^{\op})(u)\in U(1).
\]
The two representations of \(u\) imply two invariance properties. 
For all $x,y \in [0,1]^{\N_{>0}}$, if
\(\delta(x)=\delta(y)\), then
\[
  a_x =U(x^{\op})\bigl(U(\delta^{\op})(v)\bigr) =U\bigl((\delta\circ x)^{\op}\bigr)(v) =U\bigl((\delta\circ y)^{\op}\bigr)(v) = a_y.
\]
Here \(x^{\op}\circ \delta^{\op}=(\delta\circ x)^{\op}\).  Likewise, if
\(\pi_n(x)=\pi_n(y)\), then \eqref{eq:finite-stage} gives
\[
 a_x =U(x^{\op})\bigl(U(\pi_n^{\op})(w)\bigr)=U\bigl((\pi_n\circ x)^{\op}\bigr)(w) =U\bigl((\pi_n\circ y)^{\op}\bigr)(w) = a_y.
\]
Here \(x^{\op}\circ\pi_n^{\op}=(\pi_n\circ x)^{\op}\).

Applying Lemma~\ref{lem:zigzag} to these two invariances yields
\[
  a_{\underline 0}
  = a_{\underline 1}.
\]
Since \(\delta\circ\underline 0=e_0\) and
\(\delta\circ\underline 1=e_1\), this equality says
\[
  U(e_0^{\op})(v)=U(e_1^{\op})(v),
\]
contrary to \eqref{eq:distinguish}.  Therefore
\(U(\delta^{\op})(v)\notin\im(\kappa_U)\), proving both assertions.
\end{proof}

\begin{corollary}\label{cor:not-fc}
The category \(\CompHaus^{\op}\) is not finitely concrete.
\end{corollary}

\begin{proof}
The sequence \eqref{eq:diagram} is a directed diagram.  A faithful functor
preserving all directed colimits would make \(\kappa_U\) bijective, whereas
Theorem~\ref{thm:main} shows that it is not surjective.
\end{proof}

\section*{Declaration of generative AI and AI-assisted technologies}

During the development of this work, I used OpenAI's ChatGPT
(accessed in July 2026) to seek a direct topological proof that the category \(\CompHaus^{\op}\) is not finitely concrete.
ChatGPT produced the fixed-diagram formulation and the fiber-zigzag proof of Theorem~\ref{thm:main}, and assisted with drafting the manuscript.
I checked the argument in detail and revised the exposition.
I take full responsibility for the content
of the article.

\end{document}